\documentclass[12pt,leqno]{article}
\usepackage{amsmath,amstext, amsthm,amsfonts,amssymb,amsthm}
 \numberwithin{equation}{section}
\usepackage{epsfig}
\newtheorem{thm}{Theorem}[section]
\newtheorem{lem}[thm]{Lemma}

\newtheorem{cor}[thm]{Corollary}

\newtheorem{prop}[thm]{Proposition}
\newcommand{\be}{\begin{equation}}
\newcommand{\ee}{\end{equation}}

\newcommand{\ba}{\begin{array}}
\newcommand{\ea}{\end{array}}

\newcommand{\al}{\alpha}
\newcommand{\bt}{\beta}

\newcommand{\la}{\lambda}

\newcommand{\bea}{\begin{eqnarray}}
\newcommand{\eea}{\end{eqnarray}}

\newcommand{\Sum}{\sum_{n=0}^\infty}
 \renewcommand{\)}{\right) }

\title{Plancherel-Rotach Asymptotics for \\ $q$-Orthogonal Polynomials }
\author{  Mourad E.H. Ismail\thanks{The research of Mourad E.H. Ismail is supported by the Research Grants
Council of
Hong Kong  under contract \# 101411 and the  NPST Program of King Saud
University, project number 10-MAT1293-02.}  \; and    Xin  Li}

\begin{document}
 \maketitle
 \date{}
\begin{abstract}
 We establish the Plancherel-Rotach-type asymptotics around the largest zero (the soft edge asymptotics) for some classes of polynomials satisfying three-term recurrence relations with
exponentially increasing coefficients. As special cases, our results include this type of asymptotics for $q^{-1}$-Hermite polynomials of Askey, Ismail and Masson,
$q$-Laguerre polynomials, and the Stieltjes-Wigert polynomials. We also introduce a one parameter family of solutions to the $q$-difference equation of the Ramanujan function.
 \end{abstract}



\bigskip

\noindent  Filename: PlancherelRotachRevised
\bigskip

\section{Introduction}
 Using the explicit representations and $q$-identities (in particular, the $q$-binomial theorem), Ismail~\cite{ismail2005}
derived the complete asymptotics expansion of $q^{-1}$-Hermite, $q$-Laguerre, and Stieltjes-Wigert polynomials near their respective largest zeros. One of
the main features is that the Ramanujan function $A_q(z)$, called the $q$-Airy function in \cite{ismail2005}, appears very naturally in all the
results. These asymptotics are called the soft edge asymptotics in random matrices. There are two other types of Plancherel-Rotach asymptotics, the
bulk scaling asymptotics where you normalize by the largest zero but keep $x$ in the oscillatory region, and the tail asymptotics which is
asymptotics beyond the largest and smallest zeros \cite{Sze}.  Ismail and Zhang \cite{Ism:Zha1}--\cite{Ism:Zha2} derived
these asymptotics and found
that the sine and cosine functions which appear in the  bulk scaling asymptotics of the Hermite and Laguerre polynomial asymptotics \cite{Sze} are
now replaced by theta functions.

In this work, we will start with the more general three-term recurrence relations that include these classical $q$-orthogonal polynomials as special cases and
derive the Plancherel-Rotach type asymptotics. A general characteristics of these recurrence relations is that the recurrence coefficients are
growing exponentially with $n$, the degree of the polynomials. The orthogonality measures of these polynomials are not unique.  Through careful scaling
and transformation, these recurrence relations can be transferred into more trackable ones on which we base our study of the convergence.  It must be
emphasized that we use the recurrence relation directly in contrast to the previous works, \cite{ismail2005}, \cite{Ism:Zha1}--\cite{Ism:Zha2}, which
use more detailed  information, such as explicit formulas.

The approach of deriving the asymptotic properties (e.g., ratio
asymptotics, zero distributions, etc.) of orthogonal polynomials
from assumptions on the asymptotic behavior of recurrence
coefficients can be traced back to the work of
Poincar\'e~\cite{poi} and  Blumenthal~\cite{blum}. See the
monograph in the Memoirs of AMS by Nevai~\cite{nevai} for a
systematic treatment of the class $M(a,b)$ (the Nevai-Blumenthal
class as it is referred to now, see, e.g., \cite{kuij}) where the
recurrence coefficients are assumed to be convergent. Starting
with the assumptions on the recurrence coefficients, Van
Assche~\cite{van1} and Van Assche and Geronimo~\cite{vange}
obtained the Plancherel-Rotach asymptoics outside the oscillatory
region (see also~\cite{van} and the references therein). Recently,
Tulyakov~\cite{Tu} proposed a new method for obtaining the
Plancherel-Rotach type asymptotics when the recurrence
coefficients are rational functions in $n$ (indeed, he dealt with
more general difference equations of order $p$). This method
allowed him to obtain the global picture of the asymptotic
behavior of the solutions, and as a demonstration, he applied  the
method to the Hermite and Meixner polynomials.

In recent years the Riemann-Hilbert problem approach has been a
powerful technique  in determining asymptotics of orthogonal
polynomials, see \cite{Dei} for the description of the
Riemann-Hilbert approach and \cite{Dei:Kri:McL:Ven:Zho} for the
application to derive asymptotics of polynomials orthogonal with
respect to exponential weights. So far the Riemann-Hilbert problem
has not been successfully applied  to $q$-orthogonal polynomials
except in the case of the Stieltjes-Wigert polynomials when $q =
e^{-1/(4n^2)}$ and $n$ is the degree of the polynomial, see
\cite{Bai:Sui}.

Throughout this work, we will assume that $0<q<1$.

Our model case is the $q^{-1}$-Hermite polynomials. Recall that the $q^{-1}$-Hermite polynomials $\{h_n(x~|~q)\}$ of Askey~\cite{Askey7} and Ismail and
Masson~\cite{ismailmasson} satisfy the recurrence
relation
\begin{equation}\label{eqqhermite}
2xh_n(x~|~q)=h_{n+1}(x~|~q)+q^{-n}(1-q^n)h_{n-1}(x~|~q),
\end{equation}
with $h_0(x~|~q)=1$ and $h_1(x~|~q)=2x$.  For detailed properties of the $q^{-1}$-Hermite polynomials see \cite[Chapter 21]{Ism}.
Using  a result in
\cite{ismailli} it is easy to see  that the largest zero of $h_n(x~|~q)$ is less than  (but close to) $q^{-n/2}$.   Let
\bea \label{eqxnt}
x_n(t)=\frac{1}{2}\left[q^{-n/2}t- q^{n/2}/t \right]. \eea Clearly, $x_{n\pm 1}(q^{\pm 1/2}t)=x_n(t)$. Set $x=x_n(t)$
in \eqref{eqqhermite}  to
get
$$ \left[q^{-n/2}t-\frac{q^{n/2}}{t}\right]h_n(x_n(t)~|~q)=h_{n+1}(x_{n+1}(q^{1/2}t)~|~q)+q^{-n}(1-q^n)h_{n-1}(x_{n-1}(q^{-1/2}t)~|~q).
$$
With $p_n(t)=t^{-n}q^{n^2/2}h_n(x_n(t)~|~q)$, the above recurrence becomes
\begin{equation}\label{eqhermite2}
\left(1-\frac{q^n}{t^2}\right)\; p_n(t)=p_{n+1}(q^{1/2}t)+\frac{1-q^n}{t^2}p_{n-1}(q^{-1/2}t).
\end{equation}
If $p_n(t) \to f(t)$ uniformly on compact subsets of $\mathbb{C} \setminus\{0\}$,
 as $n \to \infty$, then $f(t)$ will be analytic in ${\mathbb C}\setminus \{0\}$ and satisfy
 \begin{equation}\label{eqtwo}
f(t)=f(q^{1/2}t)+\frac{1}{t^2}f(q^{-1/2}t)
\end{equation}
Since $p_n(t)$ is a polynomial in $1/t^2$, $f$ will have the Laurent expansion:
\bea
\label{eqf}
 f(t)=\sum_{n=0}^{\infty}f_nt^{-2n}, \quad \textup{with} \quad  f_0=1.
 \eea
The substitution of $f$ from \eqref{eqf} in \eqref{eqtwo}  then equating coefficients of various powers of $t$ implies that \bea \notag
f(t)=A_q(t^{-2}), \eea where $A_q(z)$ is the Ramanujan function \cite{Ram} which plays the role of  Airy function, as pointed out in
\cite{ismail2005}. The $q$-Airy function $A_q$  satisfies $q$-difference equation:
\bea
\label{eqstarx} A_q(z)=A_q(qz)-qzA_q(q^2z),
\quad A_q(0)=1 \eea
with power series expansion \bea \label{eqDefAq} A_q(z)= 1 + \sum_{k=1}^{\infty}\frac{q^{k^2}(-z)^k}{(1-q)(1-q^2)\cdots (1-q^k)}. \eea

The purpose of this paper is to justify the  above procedure  for a more general class of orthogonal polynomials. More precisely  we consider
 the family of orthogonal polynomials $\{p_n(x,c)\}$ defined by the recurrence relations
\begin{equation}\label{eqqhermite3}
2xp_n(x,c)=p_{n+1}(x,c)+q^{-nc}\beta_n(q,c)p_{n-1}(x,c),
\end{equation}
where $c>0$ and $\beta_n(q,c)=1+ o(1)$ as $n\to\infty$. We will also
remark on the case when $\bt_n (q,c)= 1+
\sum_{j=1}^{\infty}a_j(c)q^{\lambda_j n}$ for a sequence
$\{\lambda_j\}$ with $\la_{j+1} > \la_j >  0, j=1,2,3,...$, in Section~5. This
generalizes the $q^{-1}$-Hermite model. A second result
generalizes the Stieltjes-Wigert and $q$-Laguerre models. We
consider the three term recurrence relation
\begin{equation}\label{sw} xq^{2n+\alpha+1}p_n(x)=a_np_{n+1}(x)+b_np_{n-1}(x)+c_np_n(x),
\end{equation}
where, modulo linear scaling of $x$ by $ax+b$,
\begin{equation}\label{an}
a_n=1+\sum_{k=1}^{\infty}a_{n,k}q^{\alpha_k n}, ~~ b_n=1+\sum_{k=1}^{\infty}b_{n,k}q^{\beta_k n}, ~~ c_n=1+\sum_{k=1}^{\infty}c_{n,k}q^{\gamma_k n},
\end{equation}
where $\alpha_k>0$, $\beta_k>0$, and $\gamma_k>0$ for all $k>0$. We will see that this recurrence includes both the $q$-Laguerre and Stieltjes-Wigert
polynomials as special cases in next section where we will also introduce a one parameter family of solutions to the $q$-difference equation in \eqref{eqstarx}.

\section{Main Results}
As in the case of \eqref{eqqhermite} we transform  the polynomials  $p_n(x,c)$ to the functions $s_n(t)$ defined as
$s_n(t)=q^{n^2/2}t^{-n}p_n(x_n(t),c)$ with $x_n(t)$ given by \eqref{eqxnt}. Thus  $s_{-1}(t)=0$, $s_0(t)=1$, and
\begin{equation}\label{eqone}
\left(1-\frac{q^n}{t^2}\right)\, s_n(t)=s_{n+1}(q^{1/2}t)+\frac{1}{t^2}c(q,n)s_{n-1}(q^{-1/2}t),~n \geq  0,
\end{equation}
where $|c(q,n)|\leq K$ and, for each $q$,
$\lim_{n\to\infty}c(q,n)=1$.

Note that $\{s_n(t)\}$ are polynomials in $1/t^2$ of degree $\lfloor{n/2}\rfloor$.
%If we can establish the convergence of $\{s_n(t)\}$ and if we let the limiting
%function be $f(t)$.

Our main result is the proof of the convergence of $\{s_n(t)\}$. We remark that the recurrence \eqref{eqone} indeed
includes the recurrence relation of the $q^{-1}$-Hermite \eqref{eqhermite2} as a special case.

\bigskip

\begin{thm}\label{thm1}  Let $\{s_n(t)\}$ be defined by the recurrence given by
\eqref{eqone} with $s_{-1}(t)=0$ and $s_0(t)=1$. The limit $\lim_{n\to\infty}s_n(t)$ exists and
the convergence is locally uniform for $t\in \overline{{\mathbb C}}\setminus \{0\}$.
\end{thm}

By the discussion immediately following \eqref{eqhermite2}, we have the following corollary.

\begin{cor} Let $\{s_n(t)\}$ be defined by the recurrence given by \eqref{eqone}. We have
$$\lim_{n\to\infty}s_n(t)=A_q(\frac{1}{t^2})$$
locally uniform for $t\in \overline{{\mathbb C}}\setminus \{0\}$.
\end{cor}

The main ideas used in the proof of Theorem~\ref{thm1} for the treatment of the family that generalizes the $q^{-1}$-Hermite polynomials can be
adapted to the family of recurrence relations (i.e., (\ref{sw})) which include $q$-Laguerre and Stieltjes-Wigert polynomials as special cases. Due to
the nature of non-symmetry, the implementation of these ideas become lengthier in the new situation. To transform (\ref{sw}), set
$$x=x_n(t)=q^{-2n-\alpha}t$$ and define $S_n(t)=q^{n^2}(-t)^{-n}p_n(x_n(t))$. Then the recurrence relation (\ref{sw}) can be written as
\begin{equation}\label{eqsw2}
S_n(t)=(1-q^{n+1})S_{n+1}(q^2t)+(1+q)c(q,n)(qt)^{-1}S_n(t)+qd(q,n)t^{-2}S_{n-1}(q^{-2}t),
\end{equation}
where $|c(q,n)|\leq K$ and $|d(q,n)|\leq K^2$ with
$\lim_{n\to\infty}c(q,n)= \lim_{n\to\infty}d(q,n)=1$.  It is clear
that $S_n(t)$ is a polynomial in $1/t$ with \bea \label{eqsw6}
S_n(\infty) = 1/(q;q)_n. \eea

To see that $q$-Laguerre polynomials $\{L_n^{(\alpha )}(x;q)\}$
and Stieltjes-Wigert polynomials $\{S_n(x;q)\}$ are special cases
of (\ref{sw}) and (\ref{an}), we recall that $\{L_n^{(\alpha
)}(x;q)\}$ is generated by \bea \label{ql}
\begin{gathered}
-xq^{2n+\alpha+1}L_n^{(\alpha )}(x;q)=(1-q^{n+1})L_{n+1}^{(\alpha
)}(x;q)+q(1-q^{n+\alpha})L_{n-1}^{(\alpha)}(x;q) \\
-[(1-q^{n+1})+q(1-q^{n+\alpha})]
L_n^{(\alpha)}(x;q)
\end{gathered}
\eea
with initial conditions
$$L_0^{(\alpha)}(x;q)=1,~~L_1^{(\alpha)}(x;q)=\frac{1-q^{\alpha+1}-xq^{\alpha+1}}{1-q};$$
while the Stieltjes-Wigert polynomials $\{S_n(x;q)\}$ are defined
by
\begin{equation}\label{sw0}
-xq^{2n+1}S_n(x;q)=(1-q^{n+1})S_{n+1}(x;q)-[1+q-q^{n+1}]S_n(x;q)+qS_{n-1}(x;q),
\end{equation}
with the initial conditions $$S_0(x;q)=1,
~~S_1(x;q)=\frac{1-qx}{1-q}.$$ Recall that
$(a;q)_{\infty}=(1-a)(1-aq)\cdots (1-aq^{n-1})\cdots
=\prod_{n=0}^{\infty}(1-aq^{n-1}).$
\begin{thm} \label{thm2}
Let $\{S_n(t)\}$ be defined by \eqref{eqsw2} with initial conditions $S_0(t)=1$ and $S_1(t)=[1-q(1+q)c(q,1)]/t)/(1-q)$. Then the limit
$\lim_{n\to\infty}S_n(t)$ exists and
$$ \lim_{n\to\infty} S_n(t)= \frac{1}{(q;q)_\infty}A_q(\frac{1}{t}),$$ locally uniform for $t\in \overline{{\mathbb C}}\setminus \{0\}$.
\end{thm}

We have seen that $y=A_q(z)$ satisfies  the functional equation
\bea \label{eqqdiff} y(z) -y(qz) +qzy(q^2z) =0. \eea It is worth
noting that the family of functions \bea F_q(z;s) :=   e^{i\pi s}
\sum_{n=-\infty}^\infty \frac{(-1)^n q^{(n+s)^2}}{(q;q)_{n+s}}
z^{n+s} \eea are solutions to \eqref{eqqdiff} for all $s$.  The
notation used in the above equation is (cf. \cite{Gas:Rah,Ism})
\bea \label{qq} (a;q)_b = \frac{(a;q)_\infty}{(aq^b;q)_\infty}.
\eea Observe that $F_q(z;s)$ is periodic in $s$ of period 1, so
there is no loss of generality in taking $s$ in the strip  Re$(s)
\in [0, 1)$. Indeed, $A_q(z) = F_q(z;0)$.

\section{Proof of Theorem~\ref{thm1}}

Let $\{s_n(t)\}$ be given by the recurrence relation \eqref{eqone} with $s_{-1}(t)=0$ and $s_0(t)=1$. We will need the following lemma for the proof of Theorem~\ref{thm1}.

\begin{lem}\label{lem1}   If there exist constants $\rho
>0$ and $M>0$ such that
\begin{equation}\label{three}
\max_{|t|=\rho}|s_n(t)|\leq M
\end{equation}
for all $n\geq 0$, then $ \{s_n(t)\}$ is a normal family on
$\overline{{\mathbb C}}\setminus \{0\}$.
\end{lem}

\begin{proof} Since $s_n(t)$ are polynomials in $1/t^2$,
they are all analytic on $\overline{{\mathbb C}}\setminus \{0\}$. So, by the maximal modulus principle for analytic functions, (\ref{three})
implies that
\begin{equation}\label{four}
\max_{|t|\geq\rho}|s_n(t)|\leq M
\end{equation}
for all $n\geq 0$. Next, we show that (\ref{four}) implies that
\begin{equation}\label{five}
\max_{|t|\geq q^{1/2}\rho}|s_n(t)|\leq (1+\frac{1+K}{\rho^2})M
\end{equation}
for all $n\geq 2$.

Indeed, from the recurrence relation \eqref{eqone},
$$
s_{n+1}(q^{1/2}t)=
\left(1-\frac{q^n}{t^2}\right)\; s_n(t)-\frac{1}{t^2}c(q,n)s_{n-1}(q^{-1/2}t).
$$
Thus, using (\ref{four}), for $|t|=\rho$, we have
$$
|s_{n+1}(q^{1/2}t)|\leq  \left(1+\frac{1}{|t|^2} \right)\; M+
\frac{K}{|t|^2}M= \left(1+\frac{1+K}{|t|^2}\right) \; M,
$$
for $n\geq 1$. This proves (\ref{five}). Repeating the argument above, we obtain that, for $k=1,2,...$, we have
\begin{equation}\label{six}
\max_{|t|\geq q^{k/2}\rho} |s_n(t)|\leq M\prod_{j=1}^k(1+\frac{1+K}{\rho^2q^{j-1}})
\end{equation}
for $n\geq k+1$.

Since $\lim_{k\to\infty}q^{k/2}\rho=0$, we see that the sequence
$\{s_n(t)\}$ is uniformly bounded on compact subsets of
$\overline{{\mathbb C}}\setminus \{0\}$.
\end{proof}

We need one more auxiliary result for the proof of Theorem~\ref{thm1}.

 \begin{lem}\label{lem2} We have, for $|t|\geq 1$,
\begin{equation}\label{seven}
|s_n(t)|\leq \prod_{k=0}^n(1+{q^k}) \/ A_q(-\frac{K}{|t|^2})
\end{equation}
for all $n\geq 0$.
\end{lem}

\begin{proof}
We use induction. First it is trivial to verify that (\ref{seven}) holds for $n=0$. Now, assume (\ref{seven}) is true for $n$ and let us verify that
(\ref{seven}) is also true when $n$ is replaced by $n+1$.

From \eqref{eqone}, we have
$$
s_{n+1}(t)=(1-\frac{q^{n+1}}{t^2})s_n(q^{-1/2}t)-\frac{q}{t^2}c(n,q)s_{n-1}(q^{-1}t).
$$
Thus, for $|t|\geq 1$,
\begin{eqnarray*}
|s_{n+1}(t)|&\leq&
(1+{q^{n+1}})|s_n(q^{-1/2}t)|+\frac{Kq}{|t|^2}|s_{n-1}(q^{-1}t)|\\
&\leq&(1+{q^{n+1}}) \prod_{k=0}^n(1+{q^k}) \/ A_q(-\frac{Kq}{|t|^2}) +\frac{Kq}{|t|^2} \prod_{k=0}^{n-1}(1+{q^k}) \/
A_q(-\frac{Kq^2}{|t|^2})\\
&\leq&\prod_{k=0}^{n+1}(1+{q^k}) \/ \left( A_q(-\frac{Kq}{|t|^2}) +\frac{Kq}{|t|^2}
A_q(-\frac{Kq^2}{|t|^2})\right)\\
&=&\prod_{k=0}^{n+1}(1+{q^k}) A_q(-\frac{K}{|t|^2}),
\end{eqnarray*}
where in the last equality, we have applied \eqref{eqstarx} with
$z=-K/|t|^2$. This completes the proof of the proposition.
\end{proof}

Now, we are ready for the proof of Theorem~\ref{thm1}.

\begin{proof}[Proof of Theorem~\ref{thm1}]
First, note that (i) the infinite product in (\ref{seven}) converges as $n\to\infty$, (ii) $ A_q(-\frac{K}{|t|^2})>0$, and (iii) when $|t|=1$,
$$ A_q(-\frac{K}{|t|^2})\leq A$$ for some $A>0$. Hence, with the help of Lemmas~\ref{lem1} and \ref{lem2},
we see that $\{s_n(t)\}$ is a normal family for $t\in
\overline{\mathbb C}\setminus \{0\}$.

To establish the convergence of $\{s_n(t)\}$, we study the
coefficients of $s_n(t)$. Recall that $s_n(t)$ is a polynomial of
degree $\lfloor n/2\rfloor$ in $1/t^2$. Write
\begin{equation}\label{star}
s_n(t)=\sum_{k=0}^{\lfloor n/2\rfloor}
a_{n,k}\frac{1}{t^{2k}}.\end{equation} Then, the fact that
$\{s_n(t)\}$ is a normal family implies that $\{a_{n,k}\}$ is a
bounded set. Assume that \begin{equation}\label{2asts}
|a_{n,k}|\leq M, ~~~{\rm for}~~0\leq k\leq \lfloor n/2\rfloor,
~n=0,1,2,... .
\end{equation}
Next, we show that, for every fixed $k$,
\begin{equation}\label{ast}
\lim_{n\to\infty}a_{n,k}=:f_k
\end{equation}
exists.

For $n\geq 0$, using \eqref{star} in \eqref{eqone}, we get
$$
(1-\frac{q^n}{t^2})\sum_{k=0}^{\lfloor
n/2\rfloor}a_{n,k}\frac{1}{t^{2k}}= \sum_{k=0}^{\lfloor
(n+1)/2\rfloor}a_{n+1,k}\frac{1}{t^{2k}q^k}+ \frac{1}{t^2}
c(q,n)\sum_{k=0}^{\lfloor (n-1)/2\rfloor}a_{n-1,k}\frac{q^k}{t^{2k}}
$$
Comparing the constant terms on both sides of the above equation, we
obtain
$$a_{n,0}=a_{n+1,0}, ~~~n=0,1,2,... .$$
From $s_0(t)=1$, we get $a_{0,0}=1$, so, we have
\begin{equation}
\label{2stars} a_{n,0}=1, ~~n=0,1,2,... .
\end{equation}
Next, for $k>0$, comparing the coefficients of $1/t^{2k}$, we get,
for large $n$,
$$
a_{n,k}-q^na_{n,k-1}=\frac{a_{n+1,k}}{q^k}+q^{k-1}c(q,n)a_{n-1,k-1}
$$
Or, equivalently, for large $n$,
\begin{equation}
\label{3stars}
q^ka_{n,k}-q^{n+k}a_{n,k-1}=a_{n+1,k}+q^{2k-1}c(q,n)a_{n-1,k-1}.
\end{equation}
Replacing $n$ by $n+l$ in (\ref{3stars}), we get
\begin{equation}\label{4stars}
q^ka_{n+l,k}-q^{n+l+k}a_{n+l,k-1}=a_{n+l+1,k}+q^{2k-1}
c(q,n+l)a_{n+l-1,k-1}. \end{equation} Subtracting (\ref{4stars})
from (\ref{3stars}) and re-arranging the terms, we get, for large
$n$,
$$
a_{n+1,k}-a_{n+l+1,k}=q^k(a_{n,k}-a_{n+l,k})-q^{n+k}(a_{n,k-1}-q^la_{n+l,k-1})$$
$$ ~~~~~~~~~~~~~~~~~~~~~~~~~~~
-q^{2k-1}[c(q,n)a_{n-1,k-1}-c(q,n+l)a_{n+l-1,k-1}].
$$
This, together with (\ref{2asts}) and the fact that $0<q<1$, gives
us
\begin{eqnarray}\nonumber
|a_{n+1,k}-a_{n+l+1,k}|&\leq&
q^k|a_{n,k}-a_{n+l,k}|+q^{n+k}2M+q^{2k-1}|a_{n-1,k-1}-a_{n+l-1,k-1}|
\\
\label{5stars} && +q^{2k-1}M(|c(q,n)-1|+|c(q,n+l)-1|).
\end{eqnarray}
With the help of (\ref{5stars}), we are ready to prove (\ref{ast})
by induction on $k$.

When $k=0$, equations in (\ref{2stars}) give us
\begin{equation}\label{6stars}
\lim_{n\to\infty}a_{n,0}=1.
\end{equation}

Now, assume (\ref{ast}) is true when $k$ is replaced by $k-1$.
Then, for any $\varepsilon >0$, there exists a number $N>0$ such
that
\begin{equation}\label{a-a}
|a_{n-1,k-1}-a_{n+l-1,k-1}|<\varepsilon, ~n\geq N, ~l\geq 0.
\end{equation}
From $\lim_{n\to\infty}c(q,n)=1$, we may choose $N$ large enough
to ensure that
\begin{equation}\label{c-1}
|c(q,n)-1|<\varepsilon, ~n\geq N.
\end{equation}
Thus, for $n\geq N$, (\ref{5stars}) implies
\begin{equation}
|a_{n+1,k}-a_{n+l+1,k}|\leq
q^k|a_{n,k}-a_{n+l,k}|+q^{n+k}2M+q^{2k-1}\varepsilon (1+2 M).
\label{7stars}
\end{equation}
After repeatedly using (\ref{7stars}) $m$ times, we arrive at
$$
|a_{n+1,k}-a_{n+l+1,k}|\leq
q^{(m+1)k}|a_{n-m,k}-a_{n-m+l,k}|~~~~~~~~~~~~~$$
$$~~~~~~~~~~~~~~+q^{n+k}2M \sum_{j=0}^mq^{j(k-1)}+q^{2k-1}\varepsilon
(1+2M)\sum_{j=0}^mq^{jk},
$$
$$~~~~~~~~~~~~~\leq
q^{(m+1)k}2M+q^{n+k}2M \frac{1}{1-q^{k-1}}+q^{2k-1}\varepsilon (1+2M)\frac{1}{1-q^{k}}, ~~n-m\geq N, ~l\geq 0.
$$
Take $m=\lfloor \sqrt{n}\rfloor$. Then $\lim_{n\to\infty}m=\infty$
and $\lim_{n\to\infty}(n-m)=\infty$. Therefore, for large $n$ and
for all $l\geq 0$,
$$\overline{\lim_{n\to\infty}}
|a_{n,k}-a_{n+l,k}|\leq \varepsilon
(1+2M)\frac{q^{2k-1}}{1-q^k}.$$ So,
\begin{equation}\label{9stars}
\lim_{n\to\infty}(a_{n,k}-a_{n+l,k})=0
\end{equation}
uniformly in $l\geq 0$. Thus, $\{a_{n,k}\}_{n\geq k}$ is a Cauchy
sequence. Hence, (\ref{ast}) is true for any $k=0,1,2,...$

Finally, we need to derive the convergence of $\{s_n(z)\}$.
Although this can be done by a bounded convergence argument based
on (\ref{2asts}) and (\ref{ast}), we choose the following more
elementary argument for its simplicity: Note that every
subsequence $\{s_n(t)\}_{n\in\Lambda}$ of $\{s_n(t)\}$ has a
sub-subsequence $\{s_n(t)\}_{n\in\Lambda_1}$ ($\Lambda_1\subseteq
\Lambda$) that converges locally uniformly in $\overline{\mathbb
C}\setminus \{0\}$. Assume
$$f(t):=\lim_{n\to\infty, n\in \Lambda_1}s_n(t).$$ Then $$
f^{(j)}(t)=\lim_{n\to\infty,n\in \Lambda_1}s_n^{(j)}(t).$$
Evaluating at $t=\infty$, we see that
$$f(t)=\sum_{j=0}^{\infty}f_k\frac{1}{t^{2k}}.$$
Therefore, the whole sequence $\{s_n(t)\}$ must converge to the same function $f(t)$ locally uniformly in $\overline{\mathbb C}\setminus
\{0\}$.
\end{proof}

\section{Proof of Theorem~\ref{thm2}}
We need to recall some notations from $q$-theory and solve a $q$-difference equation.
\subsection{The Functional Equation} Recall that (see (\ref{qq})) \bea (a;q)_0 = 1, \quad (a;q)_n = \prod_{j=0}^{n-1}
(1-aq^j), \eea and we note the Euler identity \cite[Page xvi, (18)]{Gas:Rah} \bea \label{eqEuler} \Sum \frac{z^n}{(q;q)_n} = \frac{1}{(z;q)_\infty},
\quad |z| < 1. \eea

We will need to solve the functional equation \bea
\label{eqfunctional} f(z) = f(zq^2) + \frac{ a q}{z^2}  f(zq^{-2})
+ a\frac{1+q}{qz} f(z), \quad a = \pm 1. \eea

\begin{prop}\label{prop1} Up to a scaling factor, the above functional equation has a solution:\\
(i) when $a=1$, $f(z)=A_q(\frac{1}{z})$, and \\
(ii) when $a=-1$, $f(1/z)$ is an entire function with positive
Taylor series coefficients. \end{prop}

\begin{proof}
 Assume that $f(z) =  \Sum f_nz^{-n}$. Since $f_0 \ne 0$ is a scaling factor there is no loss of generality in assuming $f_0
=1$. By equating coefficients of $z^{-n}$ on both sides of
\eqref{eqfunctional} we see that \bea \label{eqrecurrencef} a(1-
q^{-2n})f_n = q^{2n-1} f_{n-2} + (1+q)q^{-1} f_{n-1}, \quad n >0,
\eea with $ f_1 = a q/(q-1)$. Now, in order to normalize the
coefficients in difference equation (\ref{eqrecurrencef}) above,
set $f_n = g_n q^{n^2}$. Thus $g_0 =1, g_1 = a/(q-1)$ and \bea
g_n(q^{2n}-1) = aqg_{n-2} + a (1+q) g_{n-1}, \quad n >0. \notag
\eea We now solve the above recursion using generating functions,
so we set $G(z) = \Sum g_n z^n$. In view of the initial
conditions, the functional equation for $G(z)$ is \bea
\label{eqgfunctional} G(z) = G(q^2 z)/[1+a (1+q)z + aqz^2]. \eea
If $a =1$ then the denominator in \eqref{eqgfunctional} factors as
$(1+z)(1+qz)$, and iterating \eqref{eqgfunctional} with the initial
condition $G(0) =1$ gives \bea G(z) = \frac{1}{(-z;q)_\infty} =
\sum_{k=0}^\infty \frac{(-z)^k}{(q;q)_k}, \notag \eea by the Euler
identity. In the last step we used \eqref{eqEuler}. This leads to
$f(z) = A_q(1/z)$.

On the other hand when $a=-1$, the denominator in
\eqref{eqgfunctional} factors as \bea 1-  (1+q)z - qz^2 =
(1-z/\al)(1-z/\bt), \quad \textup{with} \quad \al < 0 < \bt ~~{\rm
and}~~\bt <|\al|. \notag \eea In this case \bea G(z) =
\frac{1}{(z/\al;q^2)_\infty(z/\bt;q^2)_\infty}. \notag \eea Since
$0 < \bt < |\al|$, we see that the singularity of $G$ with the
smallest modulus is $\bt$. Applying Darboux's asymptotic method,
\cite{Olv}, we conclude that \bea g_n =
\frac{\bt^{-n}}{(q^2;q^2)_\infty(\bt/\al;q^2)_\infty}\; [1+ o(1)].
\notag \eea Thus $f_n = g_n q^{n^2}$ and $f(1/z)$ is an entire
function of $z$. Moreover the recurrence relation
\eqref{eqrecurrencef} and the initial conditions show that $f_n >
0$. Therefore the function $f(1/z)$ is an entire function
of $z$ with positive Taylor series coefficients.
\end{proof}

\begin{proof}[Proof of Theorem~\ref{thm2}]
We follow the ideas used in the proof of Theorem~\ref{thm1}. We will skip similar argument. The main task is to show that the family $\{S_n(t)\}$
defined by \eqref{eqsw2} is a normal family for $t\in \overline{{\mathbb C}}\setminus \{0\}$.

Let $f^b(t)$ be the solution with $f^b(\infty)=1$ in Proposition~\ref{prop1} when $a=-1$. We claim that for $|t|\geq 1$,
\begin{equation}\label{swproof1}
|S_n(t)|\leq \frac{f^b(|t|/K)}{\prod_{k=1}^n(1-q^k)},
\end{equation}
for $n=0,1,2,...$

Let us use induction to verify (\ref{swproof1}) for $n=0,1,2,...$
Clearly, from the initial conditions and the fact that
$f^b(|t|/K)=1+Kq/((1-q)|t|)+$(positive terms), we see that
(\ref{swproof1}) is true when $n=0,1$. Now assume that
(\ref{swproof1}) is true for $n$.

From \eqref{eqsw2}, we have
$$
(1-q^{n+1})S_{n+1}(t)=S_n(q^{-2}t)-(1+q)c(q,n)(q^{-1}t)^{-1}S_n(q^{-2}t)-q^5d(q,n)t^{-2}S_{n-1}(q^{-4}t).
$$
So, for $|t|\geq 1$,
\begin{eqnarray*}
|(1-q^{n+1})S_{n+1}(t)|&\leq &|S_n(q^{-2}t)|+\frac{(1+q)q}{|t|}|S_n(q^{-2}t)|+\frac{q^5}{|t|^{2}}|S_{n-1}(q^{-4}t)|\\
&\leq & \frac{f^b(q^{-2}|t|/K)}{\prod_{k=1}^n(1-q^k)}\left(1+\frac{K(1+q)q}{|t|}\right)+\frac{K^2q^5}{|t|^{2}}\frac{f^b(q^{-4}|t|/K)}{
\prod_{k=1}^{n-1}(1-q^k)}\\
&\leq & \frac{1}{\prod_{k=1}^n(1-q^k)}\left[\left(1+\frac{K(1+q)q}{|t|}\right) f^b(\frac{q^{-2}|t|}{K})+\frac{K^2q^5}{|t|^{2}}f^b(\frac{q^{-4}|t|}{K})\right]\\
&=&\frac{1}{\prod_{k=1}^n(1-q^k)}f^b(\frac{|t|}{K}),
\end{eqnarray*}
where in the last step, we used the functional equation
((\ref{eqfunctional}) with $a=-1$) satisfied by $f^b$ with
$z=q^{-2}|t|/K$ and the fact that $f^b(\frac{1}{z})$ has positive
Taylor series coefficients. This implies (\ref{swproof1}) holds
when $n$ is replaced by $n+1$. Therefore, by induction,
(\ref{swproof1}) is true for all $n=0,1,2,...$

The rest of the proof goes like the one for Theorem~\ref{thm1}: From the claim, we can easily show that $\{S_n(t)\}$ is a normal family for $t\in
\overline{{\mathbb C}}\setminus \{0\}$ and from here and working with the recurrence relations, we can show (by a similar argument as given in the
proof of Theorem~\ref{thm1}) that the coefficients of $S_n(t)$ converge: if $$S_n(t)=\sum_{k=0}^n a_{n,k}t^{-k}$$ then $$\lim_{n\to\infty}
a_{n,k}=f_k, ~~{\rm for}~ {\rm some}~ f_k, ~k=0,1,2,...,$$ which immediately yields the convergence of the whole sequence $\{S_n(t)\}$.

Finally, writing the limit as $f(t)$: $f(t):=\lim_{n\to\infty}S_n(t)$, and let $n\to\infty$ on both sides of
\eqref{eqsw2} to get
$$
f(t)=f(q^2t)+(1+q)(qt)^{-1}f(t)+qt^{-2}f(q^{-2}t).
$$
From this, with the fact that $\lim_{n\to\infty}S_n(\infty)=1/(q;q)_\infty$, Proposition~\ref{prop1} when $a=1$ implies that $f(t)=A_q(\frac{1}{t})/(q;q)_\infty$.
\end{proof}

\section{Error  Terms}
Unlike the asymptotics  of the $q^{-1}$-Hermite polynomials
\cite{ismail2005}, the terms beyond the main term in the
asymptotic expansion of $s_n(t)$  seem to be intricate.

 In \eqref{eqone} we assume that $c(q,n)$ has the convergent expansion
\bea c(q,n) = \sum_{k=0}^m  c_k q^{nk} + o(q^{nm}), \quad c_0 =1.
\eea We further assume that $s_n(t)$ has the asymptotic expansion
\bea \label{eqasymsn} s_n(t) =    \sum_{k=0}^m f_{n,k}(1/t^2)
q^{nk} +\textup{ o}(q^{nm}). \eea Now substitute for $s_n$ from
\eqref{eqasymsn} in \eqref{eqone} and equate the coefficients of
$q^{nk}$. When $k=0$ we  see that $f_{n,0}(1/t^2)$ solves \bea
\label{eq2ndorderqAiry} y(u) - y(u/q) - u y(qu) =0, \eea with  $u
= 1/t^2$. When $ k > 0$ we conclude that \bea \label{eqfuneqfk}
\begin{gathered}
\quad f_{n,k}(u) - q^kf_{n+1, k}(u/q) - u q^{-k}f_{n-1,k}(qu)  \qquad \qquad \\
\qquad \qquad  = uf_{n, k-1}(u) +
u \sum_{j=0}^{k-1} c_{k-j}q^{-j}f_{n-1,j}(qu),
\end{gathered}
\eea for $k=1, 2, \cdots$. Since $s_n(t)$ is a polynomial in
$1/t^2$ we expect $f_{n,k}(u)$ to be analytic in $u$ in a
neighborhood of $u=0$. Thus $f_{n,0}(u) = A_q(u)$.   Note that
  \eqref{eqfuneqfk} implies $f_{n,k}(0)=0$ for $k >0$, since $s_n(\infty) =1= A_q(0)$.
Let $g_{n,k}(u) = f_{n,k}(u)/u$, for $k >0$.  Thus   \eqref{eqfuneqfk}  gives
\bea
\label{eqgn}
g_{n,1}(u) - g_{n+1, 1}(u/q) -  ug_{n-1,1}(qu) = A_q(u) +c_1 A_q(qu).
\eea

The case $c(q,n) = 1- q^n$ is very exceptional. In this case Ismail \cite{ismail2005} showed that
\bea
\label{eqIsmasymp}
s_n(t) = \sum_{k=0}^\infty \frac{q^{j(j+1)/2}}{(q;q)_jt^{2j}}\; A_q(q^j/t^2)\; q^{jn}.
\eea
In this case \eqref{eqfuneqfk} becomes
\bea
\begin{gathered}
\quad f_{n,k}(u) - q^kf_{n+1, k}(u/q) - u q^{-k}f_{n-1,k}(qu)  \qquad \qquad \\
\qquad \qquad  = uf_{n, k-1}(u) -
u  q^{1-k}f_{n-1,k-1}(qu),
\end{gathered}
\eea and by induction $f_{n,k}(u)$ must have the form
$u^kg_{n,k}(u)$ and $g_{n,k}$ satisfy \bea g_{n,1}(u) - g_{n+1,
1}(u/q) -  ug_{n-1,1}(qu) = g_{n,k-1}(u) -  g_{n-1,k-1} (qu), \eea
which leads to the solution given by \eqref{eqIsmasymp}.

The analysis of (\ref{eqgn}) in general seems to be complicated. As an
illustration we first consider the case $c(q,n) = 1+ c_1q^n$.  We
then let $g_{n,1} = \sum_{j=0}^\infty \zeta_{n,j} u^j$ and
substitute it in  \eqref{eqgn}.  Thus $\zeta_{n+1,0} = \zeta_{n,0}
+ 1+c_1$, and we conclude that $\zeta_{n,0}= n \zeta_{0,0} + n
(1+c_1)$. Therefore  $\zeta_{n,0}$ is independent of $n$ if and
only if $c_1 =  -1$.  When $c_1=-1$ we are led to the expansion
\eqref{eqIsmasymp}.

For general $c_1$  we make the Ansatz $g_{n,1}(u) = n F(u) + G(u)
+ o(1)$ which  leads to $F(u) = A_q(u)$ and that  $G(u)$ solves
\bea G(u) - G(u/q) - uG(qu) = 2 A_q(u/q) + c_1A_q(qu). \eea
Writing $G(u) = \sum_{j=0}^\infty \la_j u^j$ we find the following
two-term recurrence relation for $\la_n$ \bea (-1)^j\la_j
q^{j-j^2} =\frac{1}{1-q^j}\, (-1)^{j-1}q^{j-1-(j-1)^2}\la_{j-1}+
2(-1)^j \frac{1+c_1q^{2j}}{1-q^j}, \eea which is easy to solve.
This process can be iterated but it gets complicated as we proceed
to higher $k$.

The error analysis for \eqref{eqsw2} is even harder.  Ismail \cite{ismail2005} proved
\bea
S_n(q^{-2n}t;q) =  \frac{1}{(q;q)_\infty} \sum_{s=0}^\infty  \frac{(-1)^s}{(q;q)_s}
q^{\binom{s+1}{2}} q^{ns}A_q(q^{-s}/t),
\eea
and
\bea
\begin{gathered}
\frac{q^{n^2}L_n^{(\alpha)}\(x_n(t);q\)}{(-t)^n}
=\frac{1}{(q;q)_\infty}\sum_{m=0}^\infty \frac{q^{m/2}}{(q;q)_m}
\, q^{mn} \\
\times
\sum_{s=0}^m \frac{(q;q)_m}{(q;q)_s(q;q)_{m-s}} (-1)^{m-s}
q^{s\al+ (m-2s)^2/2} A_q(q^{2s-m}/t).
\end{gathered}
\eea

Since $S_n(\infty) = 1/(q;q)_n$  we let $S_n(t) =
\frac{1}{(q;q)_n} \sum_{k=0}^\infty f_{n,k}(u) q^{nk}$ where $u =
1/t$.  Equation \eqref{eqsw2} leads to \bea \label{eqSW}
f_{n,0}(u) = f_{n+1,0}(u/q^2) + \frac{(1+q)}{q} u f_{n,0}(u) +
qu^2f_{n-1,0}(q^2u). \eea Any constant times $A_q(u)$ will solve
\eqref{eqSW} as per $(i)$ of Proposition 4.1.  In general, we have
\bea
\begin{gathered}
f_{n,k}(u) = f_{n+1,k}(u/q^2)q^k +   \frac{(1+q)u}{q} \sum_{j=0}^k c_{k-j} f_{n,j}(u) \qquad  \qquad \\
\qquad \qquad
+ qu^2 \sum_{j=0}^k d_{k-j} q^{-j} f_{n-1,j}(uq^2) - u^2 \sum_{j=0}^{k-1} d_{k-1-j} q^{-j} f_{n-1,j}(uq^2).
\end{gathered}
\eea
% Thus
% \bea
% f_{n,k}(u} = q^{-kn}
% \eea
We do not expect $f_{n,k}$ to depend on $n$.  One reason is the
following. In \eqref{eqsw2}, let $S_n(t) = \frac{1}{(q;q)_n} +
\sum_{j=1}^n a_{n,j}t^{-j}$. We then have \bea a_{n,1} =
(1-q^{n+1})a_{n+1,1}q^{-2} + \frac{1+q}{q(q;q)_n}c(q,n). \notag
\eea Thus \bea \notag q^{-2n}(q;q)_n a_{n,1} =
q^{-2n-2}(q;q)_{n+1} a_{n+1,1} +  \frac{1+q}{q}c(q,n) q^{-2n} \eea
which gives $q^{-2n}(q;q)_n a_{n,1} = (1+1/q)
\sum_{j=0}^{n-1}c(q,j) q^{-2j}$.

% This follows by writing $y(u/q^2)$ and $qu^2y(q^2u)$ in
% terms of $y(u), y(qu)$ and $y(u/q), uy(u)$, respectively using the fact that $y= A_q$
% solves \eqref{eq2ndorderqAiry}.

\noindent M. E. H. Ismail,
  City University of  Hong Kong,
Tat Chee Avenue, Kowloon, Hong Kong\\
and King Saud University,  Riyadh, Saudi Arabia\\
  email: ismail@math.ucf.edu

  \bigskip

\noindent X. Li, University of Central Florida, Orlando, Florida 32816, USA\\
email: xli@math.ucf.edu

\end{document}